\documentclass[12pt]{article}

\usepackage{amsmath}
\usepackage{amssymb}
\usepackage{amsthm}
\usepackage{bbm,bm}
\usepackage{color}
\usepackage{subfigure}
\usepackage{graphicx}
\usepackage{graphics}
\usepackage[pdftex,pagebackref,letterpaper=true,colorlinks=true,pdfpagemode=none,urlcolor=blue,linkcolor=blue,
citecolor=blue,pdfstartview=FitH,plainpages=true]{hyperref}
\usepackage[comma,authoryear]{natbib}

\newcommand{\Cov}[0]{\text{Cov}}
\newcommand{\Var}[0]{\text{Var}}

\newcommand{\R}[0]{\mathbb{R}}

\newcommand{\E}[0]{\mathbb{E}}

\newcommand{\Prob}[0]{\mathbb{P}}

\newcommand{\bigindex}[1]{\uppercase\expandafter{\romannumeral#1}}

\theoremstyle{plain}
\newtheorem{thm}{Theorem}[section]
\newtheorem{lem}[thm]{Lemma}
\newtheorem{prop}[thm]{Proposition}

\theoremstyle{definition}

\theoremstyle{remark}
\newtheorem{rmk}{Remark}

\long\def\hid#1*/{}  
\definecolor{Blue}{rgb}{0,0,1}
\definecolor{Red}{rgb}{1,0,0}
\begin{document}
\begin{center}
\large{\bf Simultaneous Inference for High Dimensional Mean Vectors}
\end{center}

\begin{center}
{By Zhipeng Lou and Wei Biao Wu}
\end{center}

\begin{center}
{Department of Statistics, University of Chicago}
\end{center}

\begin{center}
{\today}
\end{center}

\begin{center}
\abstract{
Let $X_1, \ldots, X_n\in\R^p$ be i.i.d.~random vectors. We aim to perform simultaneous inference for the mean vector $\E (X_i)$ with finite polynomial moments and an ultra high dimension. Our approach is based on the truncated sample mean vector. A Gaussian approximation result is derived for the latter under the very mild finite polynomial ($(2+\theta)$-th) moment condition and the dimension $p$ can be allowed to grow exponentially with the sample size $n$. Based on this result, we propose an innovative resampling method to construct simultaneous confidence intervals for mean vectors.
}
\end{center}

\section{Introduction}
Let $X_1, \ldots, X_n$ be i.i.d.~random vectors in $\R^p$ with mean vector $\E X_i = \mu$ and covariance matrix $\Cov (X_i) = \Sigma$. We are interested in conducting statistical inference for $\mu$ when the dimension $p$ can be comparable to or even much larger than $n$. Estimating $\mu$ by the traditional sample mean $\hat{\mu} = \sum_{i=1}^n X_i/n$, \cite{CCK2013} and \cite{CCK2014} stated that under suitable moment conditions, as $n\rightarrow \infty$ and possibly $p = p_n\rightarrow \infty$,
\begin{align}
\label{defkolmogorov}
\rho_n := \sup_{t\in\R}\big|\Prob(\sqrt{n}|\hat{\mu} - \mu|_\infty\leq t) - \Prob(|Y|_\infty\leq t)\big|\rightarrow 0.
\end{align}
where $Y = (Y_1, \ldots, Y_p)^\top\in\R^p$ is a centered Gaussian vector with $\Cov (Y) = \Sigma$ and $|x|_\infty = \max_{j\leq p}|x_j|$ is the usual $\ell_\infty$ norm for a vector $x = (x_1, \ldots, x_p)^\top \in\R^p$. Suppose we only have uniformly finite $q$-th ($q>3$) moment for each coordinate of $X_i$, i.e., there exists a constant $C > 0$ such that 
\begin{equation}
\label{eq:qmt}
\max_{j\leq p}\E |X_{ij}|^q \le C
\end{equation}
for some $q>3$. We can show that the condition
\begin{equation}
\label{polypn}
p (\log p)^{3q/2 -1} = o(n^{q/2 -1})
\end{equation}
is nearly optimal for (\ref{defkolmogorov}); see Proposition \ref{polyprop}. If (\ref{polypn}) is barely violated, then under (\ref{eq:qmt}) we can have $\rho_n\rightarrow 1$; see (\ref{eq:plarge}) and (\ref{eq:rho1}). Hence in general, the allowed dimension $p$ can be at most a power of $n$ if we use $\hat{\mu}$ for inference of $\mu$. In this paper, we propose a new approach to perform simultaneous inference for the mean vector with finite polynomial moments and show that our method applies under ultra high dimensional settings, in which $\log p = o(n^c)$ for some $c>0$.

In Section \ref{sectionGA}, we shall study properties of the truncated sample mean with element-wise truncation. A similar idea was adopted by \cite{Fan2016} in the estimation of covariance matrices. We shall establish a Gaussian approximation result for the truncated sample mean vector with uniformly finite $(2+\theta)$-th ($0<\theta\leq 1$) moments for $X_{ij}$, $1 \leq j \leq p$. As an important feature, the dimension $p$ can be as large as $ e^{o(n^{c})}$ for some $c>0$.

Equipped with the Gaussian approximation result for the truncated sample mean vector, we are ready to perform simultaneous statistical inference for $\mu$. In particular, we propose an innovative resampling method called truncated half sampling procedure to construct simultaneous confidence intervals of $\mu$ in Section \ref{sectionSCI}. As a main advantage of this method, we do not need to deal with the problem of estimating the covariance matrices which is highly nontrivial and computationally intensive in the high dimensional case, and which may require extra structural assumptions.

\section{The truncated sample mean vector and a Gaussian approximation theory}
\label{sectionGA}
Given any $\kappa>0$, let $t_\kappa : \R \rightarrow [-\kappa, \kappa]$ be the truncated function defined by $t_\kappa(x) = (x \wedge \kappa) \vee (-\kappa)$. Given the data $X_1, X_2, \ldots, X_n$, define the truncated sample mean vector
\begin{equation}
\label{definitionofmu}
\hat{\mu}_\kappa = \frac{1}{n}\sum_{i=1}^n t_\kappa(X_i),
\end{equation}
where $t_\kappa(X_i) = (t_k(X_{i1}), \ldots, t_\kappa(X_{ip}))^\top\in\R^p$. With properly chosen truncated level $\kappa$, we establish a Gaussian approximation result for $|\hat{\mu}_\kappa|_\infty$ with uniformly finite $(2+\theta)$-th moments for $X_{ij}$, $1 \leq j \leq p$.
\begin{thm}
\label{theoremGA}
Let $\mu = \E X_i = 0$. Assume there exist constants $b, M_\theta >0$ such that $\min\nolimits_{j\leq p}\E |X_{ij}|^2\geq b$ and $M_\theta = \max\nolimits_{j\leq p}\E |X_{ij}|^{2+\theta} < \infty$ for some $0< \theta\leq 1$. Further assume
\begin{equation}
\label{condpandn}
(\log p)^{4+3\theta} M_\theta^2= o(n^{\theta}).
\end{equation}
and we take $\kappa \asymp (n M_\theta/\log p)^{1/(2+\theta)}$. Then as $n\rightarrow \infty$,
\begin{equation}
\label{GAconsistency}
\rho_{n,\kappa} :=\sup_{t\in\R}|\Prob(\sqrt{n}|\hat{\mu}_\kappa|_\infty\leq t) - \Prob(|Y|_\infty\leq t)|\rightarrow 0.
\end{equation}
\end{thm}
\begin{proof}
For simplicity of notation, write $\E_0[X] = X - \E X$. Elementary calculation implies for $q=1,2$,
\begin{equation}
\label{eq:1}
\max_{j\leq p}\E|X_{ij}|^q\textbf{1}\{|X_{ij}|\geq \kappa\}\leq \kappa^{q-2-\theta}M_\theta.
\end{equation}
Recall that $t_\kappa(X_{ij}) = (X_{ij}\wedge \kappa)\vee (-\kappa)$. Then for any $j = 1, \ldots, p$, as $\kappa\rightarrow \infty$, we have
\begin{eqnarray*}
\big|\E |\E_0[t_\kappa(X_{ij})]|^2 - \E |X_{ij}|^2\big| &\leq &\E |X_{ij}|^2\textbf{1}\{|X_{ij}|\geq \kappa\} + (\E |X_{ij}|\textbf{1}\{|X_{ij}|\geq \kappa\})^2 \cr
&\leq &\kappa^{-\theta}M_\theta + \kappa^{-2(1+\theta)}M_\theta^2\rightarrow 0,
\end{eqnarray*}
which implies $\min_{j\leq p}\E |\E_0[t_\kappa(X_{ij})]|^2\geq b - o(1)>0$. Define $L_\kappa = \max_{j\leq p}\E |\E_0[t_\kappa(X_{ij})]|^3$. By a similar argument as above, for any $j = 1, \ldots, p$,
\begin{eqnarray*}
\E |t_\kappa(X_{ij})|^3 &= & \E |X_{ij}|^3\textbf{1}\{|X_{ij}|\leq \kappa\} + \kappa^3\E \textbf{1}\{|X_{ij}|\geq \kappa\} \cr
&\leq & \kappa^{1-\theta}\E |X_{ij}|^{2+\theta} \leq \kappa^{1-\theta}M_\theta.
\end{eqnarray*}
So we have $L_\kappa \leq 4\kappa^{1-\theta}M_\theta := \bar{L}_\kappa$. Let $\phi_\kappa$ be the quantity defined in (\ref{defphin}) with $\bar{L}_n$ replaced by $\bar{L}_\kappa$. We are to apply Lemma \ref{theoremCCk2016} for the i.i.d. vectors $\E_0[t_{\kappa}(X_i)]$, $1 \leq i \leq n$, and evaluate the quantities therein. Let $\kappa$ satisfy $2\kappa<\sqrt{n}/(4\phi_\kappa\log p)$. It can be easily seen that
\begin{eqnarray*}
M_{n,\E_0[t_\kappa(X)]}(\phi_\kappa) &:=& \E \Big[\max\limits_{j\leq p}|\E_0[t_\kappa(X_{1j})]|^3\textbf{1}\Big\{\max\limits_{j\leq p}|\E_0[t_\kappa(X_{1j})]|>\sqrt{n}/(4\phi_\kappa\log p)\Big\}\Big]\cr
&=&0.
\end{eqnarray*}
Let $Y_{\kappa}=(Y_{\kappa,1}, \ldots, Y_{\kappa, p})^\top\in\R^p$ be the analogue centered Gaussian vector with the same covariance matrix as $\E_0[t_{\kappa}(X_i)]$. It follows that
\begin{eqnarray*}
M_{n,Y_\kappa}(\phi_\kappa) &:=&\E \Big[\max\limits_{j\leq p}|Y_{\kappa,j}|^3\textbf{1}\Big\{\max\limits_{j\leq p}|Y_{\kappa,j}|>\sqrt{n}/(4\phi_\kappa\log p)\Big\}\Big] \cr
&\leq & \E |Y_\kappa|_\infty^3 \textbf{1}\{|Y_\kappa|_\infty>2\kappa\} \cr
& \lesssim & \Prob(|Y_\kappa|_\infty > 2\kappa)\kappa^3+ \int_{2\kappa}^\infty \Prob(|Y_\kappa|_\infty>x)x^2 dx =: \text{I}+\text{II}.
\end{eqnarray*}
As has been assumed, $\E |Y_{\kappa, j}|^2 = \E |\E_0[t_\kappa(X_{ij})]|^2$ is upper bounded by $M_\theta^{2/(2+\theta)}$. We have $\text{I} \lesssim p\kappa^3e^{-C\kappa^2/M_\theta^{2/(2+\theta)}}$. Also, elementary manipulation implies $\text{II} \lesssim p\kappa M_\theta^{2/(2+\theta)} e^{-C\kappa^2/M_\theta^{2/(2+\theta)}} + pM_\theta^{3/(2+\theta)}e^{-C\kappa^2/M_\theta^{2/(2+\theta)}}$.
Hence,
\begin{eqnarray*}
M_{n,Y_\kappa}(\phi_\kappa) &\lesssim & p\kappa^3e^{-C\kappa^2/M_\theta^{2/(2+\theta)}}+ p\kappa M_\theta^{2/(2+\theta)} e^{-C\kappa^2/M_\theta^{2/(2+\theta)}} + pM_\theta^{3/(2+\theta)}e^{-C\kappa^2/M_\theta^{2/(2+\theta)}} \cr
&\lesssim & p(n M_\theta)^{3/(2+\theta)}\exp\{-C(n^2M_\theta^{-(2+\theta)}/\log^{4+\theta}p)^{1/(2+\theta)}\log (pn)M_\theta\}\leq n^{-1},
\end{eqnarray*}
where the second line follows in view of (\ref{condpandn}). By Lemma \ref{theoremCCk2016}, we have
\begin{equation}
\label{GA1}
\rho_{n,\kappa}^* := \sup_{t\in\R}|\Prob(\sqrt{n}|\E_0[\hat{\mu}_\kappa]|_\infty\leq t) - \Prob(|Y_\kappa|_\infty\leq t)|\lesssim \Big(\frac{M_\theta^2\log^{4+3\theta}p}{n^\theta}\Big)^{1/(4+2\theta)} + \frac{1}{n}.
\end{equation}
With (\ref{GA1}), we are to consider the error bound between $|\hat{\mu}_\kappa|_\infty$ and $|\E_0[\hat{\mu}_\kappa]|_\infty$. Since $\E X_i=0$, we have
\begin{eqnarray*}
|\E \hat{\mu}_k|_\infty &=& \max\limits_{j\leq p}\big|\E (X_{ij} - \kappa)\textbf{1}\{X_{ij}>\kappa\} + \E (X_{ij} + \kappa)\textbf{1}\{X_{ij}<-\kappa\}\big| \cr
&\leq & \max\limits_{j\leq p}\E |X_{ij}|\textbf{1}\{|X_{ij}|>\kappa\}\leq \kappa^{- (1 + \theta)}M_\theta.
\end{eqnarray*}
By (\ref{eq:1}) and Lemma 2.1 in \cite{CCK2013}, for any $\delta>0$,
\begin{eqnarray}
\label{GA2}
\rho_{n,\kappa}^\diamond &:=&\sup_{t\in\R}\big|\Prob(\sqrt{n}|\hat{\mu}_\kappa|_\infty\leq t) - \Prob(\sqrt{n}|\E_0[\hat{\mu}_\kappa]\big|_\infty\leq t)| \cr
&\leq & \Prob(\sqrt{n}|\hat{\mu}_\kappa - \E_0[\hat{\mu}_\kappa]|_\infty>\delta) + \sup_{t\in\R}\Prob(|\sqrt{n}|\E_0[\hat{\mu}_\kappa]|_\infty - t|\leq \delta) \cr
&\leq & \Prob(\sqrt{n}|\E \hat{\mu}_\kappa|_\infty>\delta) + \rho_{n,\kappa}^* + \sup_{t\in\R}\Prob(||Y_\kappa|_\infty - t|\leq \delta) \cr
&\leq & \rho_{n,\kappa}^*+ 2 \sqrt{n}M_\theta\kappa^{- (1 + \theta)}\sqrt{\log p},
\end{eqnarray}
where the last inequality follows by taking $\delta = \sqrt{n}M_\theta\kappa^{- (1 + \theta)}$.\\
Next we compare $|Y_k|_\infty$ and $|Y|_\infty$. Observe that
\begin{eqnarray*}
\max\limits_{j,l}|\Cov(Y_{\kappa, j}, Y_{\kappa, l}) - \Cov(Y_j, Y_l)|
&= & \max_{j,l}|\Cov(t_\kappa(X_{ij}), t_\kappa(X_{il}) - \Cov(X_{ij}, X_{il})| \cr
&\leq & C  \kappa^{-\theta}M_\theta.
\end{eqnarray*}
By Lemma 3.1 in \cite{CCK2013}, we obtain
\begin{equation}
\label{comparison}
\rho_{n,\kappa}^\circ:=\sup_{t\in\R}\big|\Prob(|Y_\kappa|_\infty\leq t) - \Prob(|Y|_\infty \leq t)\big|\lesssim(M_\theta/\kappa^\theta)^{1/3}(1\vee \log (p\kappa^\theta))^{2/3}.
\end{equation}
Therefore, by (\ref{GA1}), (\ref{GA2}) and (\ref{comparison}), we have
\begin{eqnarray*}
\rho_{n,\kappa} &\leq& \rho_{n,\kappa}^* + \rho_{n,\kappa}^\diamond + \rho_{n,\kappa}^\circ \cr
&\lesssim & \rho_{n,\kappa}^* + \sqrt{n}M_\theta\kappa^{-(1+\theta)}\sqrt{\log p} + (M_\theta/\kappa^\theta)^{1/3}\log^{2/3} p \cr
& = & \rho_{n,\kappa}^* + \Big(\frac{M_\theta^2\log^{4+3\theta}p}{n^\theta}\Big)^{1/(4+2\theta)} + \Big(\frac{M_\theta^2\log^{4+3\theta}p}{n^\theta}\Big)^{1/(6+3\theta)} \cr
& \leq & \frac{1}{n} + 3\Big(\frac{M_\theta^2\log^{4+3\theta}p}{n^\theta}\Big)^{1/(6+3\theta)}\rightarrow 0.
\end{eqnarray*}
\end{proof}

\begin{rmk}
By Propsition \ref{polyprop}, the Gaussian approximation $\rho_n \to 0$ is valid for the traditional sample mean vector if each coordinate has finite $q$-th moment and
\begin{equation}
\label{polynomial grow}
p(\log p)^{3q/2-1}=o(n^{q/2-1}).
\end{equation}
The above condition is optimal up to a logarithmic factor. In fact, if 
\begin{eqnarray}
\label{eq:plarge}
n^{q/2-1} = o (p (\log p)^{-2-q/2}), 
\end{eqnarray}
and $X_{ij}$ are i.i.d.~symmetric random variables with $\E X_{ij} = 0$, $\E |X_{ij}|^2 = 1$ and the tail probability $\Prob(X_{ij}\geq x) = x^{-q}(\log x)^{-2}$, $x\geq x_0$, then 
\begin{eqnarray}
\label{eq:rho1}
 \rho_n \to 1; 
\end{eqnarray} see the discussion in Remark 2 of \citet{Zhang2017}. Here using the truncated sample mean, Theorem \ref{theoremGA} allows $p$ to grow exponentially with $n$. For example, with only finite third moment, i.e., if $M_1 =\max_{j \leq p} \E |X_{ij}|^3 = O(1)$, $p$ can be as large as $e^{o(n^{1/7})}$.
\end{rmk}

\begin{prop}
\label{polyprop}
Assume (\ref{eq:qmt}) and there exists some constant $b>0$ such that $\Var (X_{ij}) \geq b$ for all $j =1, \ldots, p$. Then $\rho_n \rightarrow 0$ as $n\rightarrow \infty$ under   
\begin{equation}
\label{polycond}
p (\log p)^{3q/2 -1} = o(n^{q/2 -1}).
\end{equation}
\end{prop}
\begin{proof}
We apply Lemma \ref{theoremCCk2016}. Recall the definitions of the quantities $\phi_n$, $M_{n, X}(\phi_n)$ and $L_n$ therein. Denote $\tau = \sqrt{n}/(4\phi_n \log p)$ and substitute $\phi_n=K_2(\bar{L}_n^2\log^4 p/n)^{-1/6}$ with some $\bar{L}_n\geq L_n$. We can obtain
\begin{equation}
\label{tau}
\tau = (n\bar{L}_n/\log p)^{1/3}/(4K_2)
\end{equation}
Since $\max_{j\leq q}\E |X_{ij}|^q<\infty$, by the Bonferroni technique and Markov's inequality,
\begin{equation}
\label{M}
M_{n,X}(\phi_n)
= \Prob(|X_i|_\infty>\tau)\tau^3 + 3\int_{\tau}^\infty \Prob(|X_i|_\infty>x)x^2 dx
\lesssim  p \tau^{3-q}.
\end{equation}
Choose $\bar{L}_n = \max_{j \leq p}\E|X_{1j}|^3 \cdot (p^{3/q}(\log p)^{1-3/q}n^{3/q-1})^{1-\chi}(n^{1/2}(\log p)^{-7/2})^\chi$, where $(1-3/q)/(3/2-3/q)<\chi<1$ so that $\bar{L}_n \geq L_n$ is ensured. By (\ref{tau}) and (\ref{M}), elementary calculation indicates that $\rho_n \rightarrow 0$ under (\ref{polycond}).
\end{proof}

%

\section{Construction of Simultaneous Confidence Intervals}
\label{sectionSCI}
Based on the Gaussian approximation result for the truncated sample mean (cf. Theorem \ref{theoremGA}), we are able to construct simultaneous confidence intervals (SCIs) of $\mu$. Given a confidence level $\alpha\in(0,1)$, the $(1-\alpha)$ SCIs can be constructed by
\begin{align}
\label{SCIform}
\mathcal{C}_{\alpha, \kappa}= \big\{\nu = (\nu_1, \ldots, \nu_p)^\top \in\R^p : \big|\sum_{i=1}^n t_\kappa(X_i - \nu)\big|_\infty \leq \sqrt{n}c_{1-\alpha}\big\}
\end{align}
where $c_{1-\alpha}$ is the cutoff value determined by (\ref{cutoff1}) below. Note that for every $j = 1, \ldots, p$, $f_j(y) := \sum_i t_\kappa (X_{ij} - y)$ is a non-increasing and continuous function of $y$ and it is lower and upper bounded by $-n\kappa$ and $n\kappa$ respectively. Assume $n \kappa > \sqrt{n}c_{1-\alpha}$. Let $l_j, u_j$ be the solutions to the equation $f_j(y) = \pm \sqrt{n}c_{1-\alpha}$. Then the SC for $\mu_j$ is $[l_j, u_j]$, $1 \le j \le p$.

Note that $|\sum_{i} t_\kappa(X_i - \nu)|_\infty = 0$ gives the Huber estimate.

Similarly, the null hypothesis $H_0: \mu = \mu_0$ can be rejected at level $\alpha$ if $|\sum_{i} t_\kappa(X_i - \mu_0)|_\infty \geq \sqrt{n}c_{1-\alpha}$.

Now we shall determine the cutoff value $c_{1-\alpha}$ so that the SCIs given by (\ref{SCIform}) is asymptotically valid. If $\Cov (X_i)=\Sigma$ is known, by Theorem \ref{theoremGA}, the result
\begin{equation}
\label{GAuse}
\sup_{t\in\R}\Big|\Prob\Big(\frac{1}{\sqrt{n}}\Big|\sum_{i=1}^n t_\kappa(X_i - \mu)\Big|_\infty\leq t\Big) - \Prob(|Y|_\infty\leq t)\Big|\rightarrow 0
\end{equation}
implies that we can compute $c_{1-\alpha}$ by
\begin{align}
\label{cutoff1}
c_{1-\alpha} = \inf \{t\in\R: \Prob(|Y|_\infty \leq t)\geq 1-\alpha\}.
\end{align}
When $\Sigma$ is unknown, one natural approach is to use a consistent estimator. Estimation of covariance matrices in high dimensions is highly nontrivial. The restriction mainly lies in the requirement of extra structural assumptions such as bandedness or sparsity.

Now we propose an innovative method to estimate the cutoff value $c_{1-\alpha}$ which can avoid the estimation of covariance matrices and is very convenient to implement. For simplicity, suppose the sample size is even denoted by $n = 2m$. Let $\pi = (\pi(1), \ldots, \pi({n}))$ be a permutation of $[n] = \{1, 2, \ldots, n\}$ and define $Z_{\pi, i} = (X_{\pi(i)} - X_{\pi(m+i)})/\sqrt{2}$ for $i = 1, \ldots, m$. By such construction, $Z_{\pi, 1},\ldots, Z_{\pi, m}$ are then i.i.d.~symmetric random vectors with covariance matrix $\Sigma$. By Theorem \ref{theoremGA} and the fact that $\E t_\kappa(Z_{\pi, i}) = 0$, we have
\begin{align}
\label{GAofZ}
\sup_{t\in\R}\big|\Prob(\sqrt{m}|\bar{Z}_{\pi, \kappa}|_\infty\leq t)- \Prob(|Y|_\infty\leq t)\big|\rightarrow 0,
\end{align}
where $\bar{Z}_{\pi, \kappa} = \sum_{i = 1}^m Z_{\pi, i}/m$. The result (\ref{GAofZ}) provides theoretical support for the validity of the procedure estimating $c_{1-\alpha}$ by the $(1-\alpha)$ quantile of $\sqrt{m}|\bar{Z}_{\pi,\kappa}|$. Let $\Pi_{n}$ be the collection of all permutations of $[n]$ and let $\pi_1, \ldots, \pi_J$ be i.i.d. uniformly sampled from $\Pi_{n}$. Assume that
the sampling process $(\pi_i)_{i \geq 1}$ is independent of $(X_i)$. Define
\begin{align}
\label{resampledistribution}
F_J(t) = \frac{1}{J}\sum_j^J \textbf{1}\{\sqrt{m}|\bar{Z}_{\pi_j ,\kappa}|_\infty\leq t\},
\end{align}
where $\bar{Z}_{\pi_j, \kappa} = \sum_{i=1}^m Z_{\pi_j, \kappa}/m$. We can obtain the empirical $(1-\alpha)$-th quantile of $F_J(t)$ by $\hat q_{1-\alpha} = \inf\{t\in\R : F_J(t)\geq 1-\alpha\}$ and estimate $c_{1-\alpha}$ by $\hat q_{1-\alpha}$.  The simultaneous confidence interval for $\mu$ can be constructed as
\begin{align}
\mathcal{C}_{\alpha, \kappa}^\dag= \big\{\nu\in\R^p : \big|\sum_{i} t_\kappa(X_i - \nu)\big|_\infty \leq \sqrt{n} \hat q_{1-\alpha}\big\}.
\end{align}


\section{Appendix}
Lemma \ref{theoremCCk2016} follows from Theorem 2.1 in \cite{CCK2014} for the i.i.d. case, which is useful to prove Theorem \ref{theoremGA}.
For the i.i.d. random vectors $X_i$, define $L_n = \max\limits_{j\leq p}\E |X_{ij}|^3$ and for $\phi>1$, define
\begin{equation}
\label{defMX}
M_{n,X}(\phi) = \E \Big[\max\limits_{j\leq p}|X_{1j}|^3\textbf{1}\Big\{\max\limits_{j\leq p}|X_{1j}|>\sqrt{n}/(4\phi\log p)\Big\}\Big].
\end{equation}
Let $Y$ be a centered Gaussian vector with $\Cov (Y) = \Sigma$. We can define $M_{n,Y}(\phi)$ similarly with $X_{1j}$ replaced by $Y_{j}$ in (\ref{defMX}). Let
\begin{equation}
M_n(\phi) := M_{n,X}(\phi) + M_{n,Y}(\phi).
\end{equation}

\begin{lem}
\label{theoremCCk2016}
Suppose that there exists some constant $b>0$ such that $\min_{j \leq p} \E X_{ij}^2\geq b$. Then there exist constants $K_1, K_2>0$ depending only on $b$ such that for every constant $\bar{L}_n\geq L_n$,
\begin{align}
\rho_n\leq K_1\bigg[\Big(\frac{\bar{L}_n^2\log^7 p}{n}\Big)^{1/6} + \frac{M_n(\phi_n)}{\bar{L}_n}\bigg]
\end{align}
with
\begin{align}
\label{defphin}
\phi_n = K_2\Big(\frac{\bar{L}_n^2\log^4 p}{n}\Big)^{-1/6}.
\end{align}
\end{lem}

\bibliographystyle{unsrtnat}
\bibliography{Truncation}
\end{document}